\documentclass[leqno,11pt]{amsart}

\usepackage{amsfonts}
\usepackage{amsmath}
\usepackage{amsfonts}
\usepackage{amssymb}
\usepackage{amscd}
\newtheorem{theorem}{Theorem}
\newtheorem*{acknowledgements*}{Acknowledgements}

\newtheorem{corollary}[theorem]{Corollary}

\newtheorem*{example*}{Example}

\newtheorem{remark}[theorem]{Remark}

\newcommand{\set}[1]{{\left\{#1\right\}}}               
\newcommand{\pa}[1]{{\left(#1\right)}}                  
\newcommand{\abs}[1]{{\left|#1\right|}}                 
\newcommand{\pair}[1]{\left\langle#1\right\rangle}      
\newcommand{\erre}{\mathbb{R}}

\newcommand{\cinf}{C^{\infty}(M)}

\newcommand{\hess}{\operatorname{Hess}}

\renewcommand{\hat}[1]{\widehat{#1}}
\renewcommand{\tilde}[1]{\widetilde{#1}}

\begin{document}

\title[Quasi-Einstein manifolds and Einstein warped products]{Some triviality results for quasi-Einstein manifolds and Einstein warped products}
\author{Paolo Mastrolia}
\address{Dipartimento di Matematica\\
Universit\`a degli Studi di Milano\\
via Saldini 50\\
I-20133 Milano, ITALY}
\email{paolo.mastrolia@gmail.com}

\author{Michele Rimoldi}
\address{Dipartimento di Matematica\\
Universit\`a degli Studi di Milano\\
via Saldini 50\\
I-20133 Milano, ITALY}
\email{michele.rimoldi@unimi.it}



\begin{abstract}
In this paper we prove a number of triviality results for Einstein warped products and quasi-Einstein manifolds using different techniques and under assumptions of various nature. In particular we obtain and exploit gradient estimates for solutions of weighted Poisson-type equations and adaptations to the weighted setting of some Liouville-type theorems.
\end{abstract}

\maketitle

\section{Introduction}

A \emph{weighted manifold}, also known in the literature as a \emph{smooth metric measure space}, is a triple $(M^m,g_M, e^{-f}d\rm{vol})$, where $M^m$ is a complete $m$-dimensional Riemannian manifold with metric $g_M$, $f \in \cinf$ and $d\rm{vol}$ denotes the canonical Riemannian volume form on $M$. The Ricci tensor can be naturally extended to weighted manifolds introducing the modified \emph{$k$-Bakry-Emery Ricci tensor}
  \begin{equation}\label{kBakryEmeryRicc}
    Ric^k_f = Ric + \hess(f) - \frac 1k df\otimes df, \qquad \text{for}\,\, 0 <k\leq \infty.
  \end{equation}
  When $f$ is constant, $Ric^k_f \equiv Ric$, while, if $k=\infty$, $Ric^k_f = Ric_f$, the usual Bakry-Emery Ricci tensor. For a detailed introduction to weighted manifolds and the $k$-Bakry-Emery Ricci tensor, we refer to the papers of Wei and Wylie (\cite{WeiWylie2proc}, \cite{WeiWylie1}) and Li (\cite{li}).

    We call a weighted manifold \emph{$k$-quasi-Einstein} or simply \emph{quasi-Einstein} (and $g_M$ is a \emph{quasi-Einstein metric}) if
   \begin{equation}\label{QE}
Ric_f^k=\lambda g_M,
\end{equation}
    for some $\lambda \in \erre$ (see \cite{caseshuwei}).
  We note that: \begin{itemize}
                  \item if $f=$ constant, \eqref{QE} is the Einstein equation, and in this case we call the quasi-Einstein metric \emph{trivial};
                  \item if $k=\infty$, \eqref{QE} is exactly the gradient Ricci soliton equation. In the last years, since the appearance of the seminal works of R. Hamilton \cite{hamilton} and G. Perelman \cite{perelman}, the study of Ricci solitons (and of their generalizations) has become the matter of a rapidly increasing investigation, directed mainly toward problems of \emph{classification} and \emph{triviality}; among the enormous literature on the subject we only quote, as a few examples, the papers \cite{petwylie1}, \cite{petwylie2}, \cite{petwylie3}, \cite{prims}, \cite{prrims}, \cite{ELnM}.
                \end{itemize}
  In the following we deal only with the case $k \in \mathbb{N}$, which corresponds to the case of \emph{Einstein warped product metrics}. Indeed, in \cite{caseshuwei}, elaborating on \cite{kimkim}, it is proved a characterization of quasi-Einstein metrics as base metrics of Einstein warped product metrics. This characterization can be formulated in the following form (see \cite{rimoldi1}, Theorem 2). Recall that the \emph{$f$-Laplacian} of a weighted manifold $(M, g_M, e^{-f}d\rm{vol})$ is defined as the diffusion-type operator $\Delta_f=e^f\operatorname{div}(e^{-f}\nabla \,)$.

  \begin{theorem}\label{EwpandQEcharacterization}
    If $N^{m+k}=M^m \times_u F^k$ is a complete Einstein warped product with Einstein constant $\lambda$, warping function $u=e^{-f/k}$ and Einstein fibre $F^k$, then the weighted manifold $(M^m,g_M, e^{-f}d\rm{vol})$ satisfies the quasi-Einstein equation \eqref{QE}; furthermore, the Einstein constant $\mu$ of the fibre satisfies the equation
   \begin{equation}\label{EinsteinConstants}
    \Delta_f f=k\lambda-k\mu e^{\frac{2}{k}f}.
    \end{equation}
    Conversely, if the weighted manifold $(M^m,g_M, e^{-f}d\rm{vol})$ satisfies \eqref{QE}, then $f$ satisfies \eqref{EinsteinConstants} for some constant $\mu \in \erre$. Consider the warped product $N^{m+k}=M^m \times_u F^k$, with $u=e^{-f/k}$, and Einstein fibre $F$ with Einstein constant $\mu$. Then $N$ is Einstein with $Ric_N = \lambda g_N$.
  \end{theorem}

  The previous characterization  permits to study Einstein warped products by focusing only on equation \eqref{EinsteinConstants}.

Examples of quasi-Einstein manifolds with $\lambda < 0$ and $\mu$ of arbitrary sign, or with $\lambda=0$ and $\mu \geq 0$ are constructed in \cite{Besse}. Moreover, in the latter case, all non-trivial examples have $\mu >0$, while the trivial quasi-Einstein metrics with $\lambda=0$ necessarily satisfy $\mu=0$. Other non-trivial examples with $\lambda>0$, $k>1$ and $\mu>0$ are constructed in \cite{LuPagePope}. Since, if $k<\infty$ and $\lambda>0$, $M$ is necessarily compact (see \cite{qian}), the maximum principle applied to \eqref{EinsteinConstants} yields that $\mu >0$ in this situation.

The aim of this work is to prove a number of triviality results obtained with different techniques and under assumptions of various nature. The paper is organized as follows.

In Section \ref{2} the key tools are gradient estimates for solutions of the class of equations to which \eqref{EinsteinConstants} belongs to.  First we prove a new gradient estimate which extend to the case $\lambda< 0$ the one given by J. Case in \cite{case}. This will allow us to  obtain a triviality result when the function $f$ is bounded from below by a constant depending on $m$, $k$ and on the Einstein constants $\lambda$ and $\mu$, respectively of the warped product and of the fibre. Further triviality results adapting Liouville-type theorems from \cite{MasRig} are also given.

In \cite{rimoldi1} one of the authors prove a triviality result under weighted integrability conditions on $f$. In Section \ref{3}, using a Motomiya-type theorem, we are able to obtain the same conclusion under a more natural integrability assumption.

In Section \ref{4} we concentrate on 1-quasi-Einstein manifolds. Indeed, by means of an adaptation to the weighted Laplacian of a Liouville-type result obtained in \cite{PRS_JFA1}, we get the triviality of the quasi-Einstein structure or the constancy of the scalar curvature.

\section{Gradient estimates and triviality results}\label{2}

\subsection{A generalization of Case's gradient estimate}
In \cite{case} J. Case deals with the triviality of quasi-Einstein metrics, and hence, of Einstein warped products, by considering only equation \eqref{EinsteinConstants}. However in that work only the case $\lambda\geq0$ is studied.
\begin{theorem}\label{CaseTriv}(Case)
Let $N^{m+k}=M^{m}\times_{u}F^{k}$ be a complete warped product with warping function $u(x)=e^{-\frac{f\left(  x\right)  }{k}}$, scalar curvature $^NS\geq 0$ and complete Einstein fibre $F$.
Then $N$ is simply a Riemannian product provided the base manifold $M$ is complete and the scalar curvature of $F$ satisfies $^{F}S\leq 0$.
\end{theorem}
\begin{remark}\rm{
If $\lambda>0$, as observed above, the assertion of Theorem \ref{CaseTriv} can be easily proved from \eqref{EinsteinConstants} using the maximum principle.\\
Note also that in \cite{rimoldi1} a generalization of Theorem \ref{CaseTriv} is obtained, proving the triviality when the warped product and the fibers, respectively, have non-positive and non-negative scalar curvature, up to assume an integrability condition on the warping function $u$.}

\end{remark}

The proof of Theorem \ref{CaseTriv} is a consequence of the following gradient estimate for solutions of weighted Poisson equations (see also \cite{SchoenYau}).We denote with $B(q,T)$ the geodesic ball centered in $q$ of radius $T$.
\begin{theorem}\label{CaseEstimate}(Case)
Let $(M^m,g_M, e^{-f}d\rm{vol})$ be such that $Ric_f^k\geq 0$, $k<\infty$, and
\begin{equation}\label{Poisson}
\Delta_ff=\phi(f),
\end{equation}
where $\phi:\mathbb{R}\to\mathbb{R}$ is a function such that
\begin{equation}
\phi^{\prime}(t)+\frac{2}{m}\phi(t)\geq 0
\end{equation}
for all $t\in\mathbb{R}$. Then for all $q\in M$, $T>0$ such that $B(q,T)$ is geodesically connected in $M$ and the closure $\overline{B(q,T)}$ is compact,
\begin{equation}\label{GradEstCase}
|\nabla f|^2(q)\leq\frac{mk}{m+k}\frac{2(m+k+6)}{T^2}.
\end{equation}
\end{theorem}
We are able to obtain a similar estimate even in case $\lambda< 0$.
\begin{theorem}\label{estCaseStyle}
    Let $(M^m, g_M, e^{-f} d\rm{vol})$ be a weighted manifold (not necessarily complete). Suppose that, for some $k < +\infty,\, Z > 0$,
    \begin{equation}\label{condriccimfZ}
    Ric_f^k \geq\lambda= -(m+k-1)Z^2
    \end{equation}
    and that
    \begin{equation}\label{condfLaplf}
      \Delta_f f = \psi(f),
    \end{equation}
    where $\psi : \mathbb{R} \to \mathbb{R}$ satisfy
    \begin{equation}\label{psiprimepsi}
      \psi'(t) + \frac {2}{m} \psi(t) +\lambda \geq 0
    \end{equation}
    for all $t \in \mathbb{R}$. Then for all $q \in M$ and $T>0$ such that $B(q, T)$ is geodesically connected in $M$ and the closure $\overline{B_(q, T)}$ is compact,
     \begin{equation}\label{gradEstCasestyle}
      |\nabla f|^2(q) \leq  \frac{mk}{m+k} \left[\frac{2(m+k+6)}{T^2} - \frac{4\sqrt{3}}{9}\frac{\lambda}{Z}\frac{1}{T}\right].
    \end{equation}
\end{theorem}
\begin{remark}
Note that in case $Ric_f^k\geq 0$ we recover Case's result by letting $Z\to 0^{+}$.
\end{remark}
    \begin{proof}
      From the Bochner formula (see, for instance, \cite{MasRig})
      \[
      \frac{1}{2} \Delta_f |\nabla u|^2 = |Hess(u)|^2 + g_M( \nabla \Delta_f u, \nabla u) + Ric_f^k(\nabla u, \nabla u) + \frac{1}{k} g_M( \nabla f, \nabla u)^2.
      \]
      Applying the previous formula to $f$ and using \eqref{condriccimfZ}, \eqref{condfLaplf}, \eqref{psiprimepsi}, Newton inequalities  and $\Delta f = \Delta_ff + |\nabla f|^2$ we obtain
     \begin{align*}
         \frac {1}{2} \Delta_f |\nabla f|^2 &= |Hess(f)|^2 + g_M(\nabla \Delta_f f, \nabla f) + Ric_f^k(\nabla f, \nabla f) + \frac {1}{k} |\nabla f|^4 \\
          &\geq |Hess(f)|^2 + \psi'(f)|\nabla f|^2 + \lambda |\nabla f|^2 + \frac {1}{k} |\nabla f|^4 \\
           &\geq \frac {1}{m}(\Delta f)^2 + \psi'(f)|\nabla f|^2 + \lambda |\nabla f|^2 + \frac {1}{k} |\nabla f|^4 \\
           &= \frac {1}{m} \psi^2(f) + \pa{\psi'(f) + \frac {2}{m} \psi(f) + \lambda}|\nabla f|^2 + \left(\frac {1}{m} + \frac {1}{k}\right)|\nabla f|^4 \\ &\geq \left(\frac {1}{m} + \frac {1}{k}\right)|\nabla f|^4,
      \end{align*}
      and then we deduce
      \begin{equation}\label{estimDeltaffabs}
        \Delta_f |\nabla f|^2 \geq 2\left(\frac{1}{m}+\frac{1}{k}\right) |\nabla f|^4.
      \end{equation}
      Let now $\rho(x) := dist(q, x)$ (using the Calabi trick, \cite{calabi}, we can suppose that $\rho$ is smooth) and consider on $B(q, T)$ the function
      \begin{equation}
        F(x) = \left[T^2 - \rho^2(x)\right]^2 |\nabla f|^2.
      \end{equation}
      If $|\nabla f| \equiv 0$ we have nothing to prove; if $|\nabla f| \not \equiv 0$, since $F \geq 0$ and  $\left. F \right|_{\partial B(q, T)} \equiv 0$, there exists a point $x_0 \in B(q,T)$ such that $F(x_0) = \underset{\overline{B(q, T)}}\max F(x) > 0$. At $x_0$ we then have
      \begin{equation}\label{nablaFoverF}
      \frac{\nabla F}{F}(x_0) = 0,
      \end{equation}
      \begin{equation}\label{deltafFoverF}
       \frac{\Delta_f F}{F}(x_0) \leq 0.
      \end{equation}

      A long but straightforward calculation shows that \eqref{nablaFoverF} is equivalent to
      \begin{equation}\label{nablaFoverFtranslated}
      \frac{\nabla |\nabla f|^2}{|\nabla f|^2}= \frac{2\nabla \rho^2}{T^2-\rho^2} \qquad \text{ at } x_0,
      \end{equation}
      while, using \eqref{nablaFoverFtranslated} and the Gauss lemma, condition \eqref{deltafFoverF} is equivalent to
       \begin{equation}\label{deltafFoverFtranslated}
       0 \geq -2\frac{\Delta_f \rho^2}{T^2-\rho^2} + \frac{\Delta_f|\nabla f|^2}{|\nabla f|^2} -24 \frac{\rho^2}{(T^2-\rho^2)^2} \quad \text{ at } x_0.
      \end{equation}
       From the $f$-Laplacian comparison theorem (see \cite{qian}, \cite{mrs}) we have
      \begin{equation}\label{fLaplcomparisonZ}
        \Delta_f \rho^2 \leq 2\left[(m+k) + (m+k-1)Z\rho\right];
      \end{equation}
      combining \eqref{estimDeltaffabs}, \eqref{deltafFoverFtranslated} and \eqref{fLaplcomparisonZ} we find, at $x_0$,
      \begin{equation*}
        0 \geq -4 \frac{\left[(m+k) + (m+k-1)Z\rho\right]}{T^2-\rho^2} + 2  \left(\frac{1}{m}+\frac{1}{k}\right)|\nabla f|^2 -24 \frac{\rho^2}{(T^2-\rho^2)^2},
      \end{equation*}
      which implies, multiplying through by $(T^2-\rho^2)^2$, that at $x_0$ we have
      \begin{equation}
        0 \geq -4 \left[(m+k) + (m+k-1)Z\rho\right](T^2-\rho^2) + 2  \left(\frac{1}{m}+\frac{1}{k}\right)F -24 \rho^2.
      \end{equation}
      The previous relation can be rewritten as
      \begin{equation}\label{estWithH_3}
        0 \geq -4(m+k)(T^2-\rho^2) + 2  \left(\frac{1}{m}+\frac{1}{k}\right)F -24 \rho^2 + H_3(\rho),
      \end{equation}
      where $H_3 : [0, T] \to \mathbb{R}$ is defined by $H_3(\rho)=4(m+k-1)Z(\rho^3-T^2\rho)$. Since $H_3$ assumes its minimum value $-\frac{8\sqrt{3}}{9}(m+k-1)ZT^3 =$  $\frac{8\sqrt{3}\lambda}{9Z}T^3$ for $\bar{t} = \frac{T}{\sqrt{3}}$, equation \eqref{estWithH_3} implies
      \begin{equation*}
        0 \geq -4(m+k)T^2 + 2  \left(\frac{1}{m}+\frac{1}{k}\right)\left[T^2 - \rho^2(x)\right]^2 |\nabla f|^2 + \frac{8\sqrt{3}\lambda}{9Z}T^3 -24\rho^2,
      \end{equation*}
      and so
      \[
       2  \left(\frac{1}{m}+\frac{1}{k}\right)\left[T^2 - \rho^2(x)\right]^2 |\nabla f|^2 \leq 4(m+k+6)T^2 - \frac{8\sqrt{3}\lambda}{9Z}T^3,
      \]
      which easily implies the thesis taking the $\sup$ on $B(q, T)$.
\end{proof}

As a corollary we immediately get to the following Liouville-type theorem for Einstein warped products.
\begin{theorem}\label{ExtCase}
    Let $N = M^m \times_u F^k$ a complete Einstein warped product with warping function $u = e^{-f/k}$. scalar curvature $^NS=(m+k)\lambda<0$ and complete Einstein fibre $F^k$ with scalar curvature $^FS=k\mu < 0$. Suppose that
    \begin{equation}\label{condflambdamu}
      f \geq \frac {k}{2} \log\left(\frac{\lambda}{2\mu}\frac{m+2k}{m+k}\right) \qquad \text{for all } x \in M.
    \end{equation}
    Then $N$ is simply a Riemannian product (up to a rescaling of the metric on $F$).
\end{theorem}

    \begin{proof}
      Since $N$ is an Einstein warped product, from Theorem \ref{EwpandQEcharacterization}  we know that $f$ satisfies the equation
      \begin{equation*}
        \Delta_ff = k\lambda -k\mu e^{\frac{2}{k} f},
      \end{equation*}
      so, with the notation used above, we have that  $\psi(t) = k\lambda -k\mu e^{\frac {2}{k} t}$. Equation \eqref{condflambdamu} implies  \eqref{psiprimepsi}, so we can apply Theorem \ref{estCaseStyle}. Since $M$ is complete, letting $T \to +\infty$ we obtain the thesis. 
    \end{proof}

\subsection{Applications of other gradient estimates}
In the same spirit, adapting results from \cite{MasRig} we also achieve other triviality results from \textit{a-priori} estimates for the gradient of global solutions of equations slightly more general than \eqref{Poisson}. In particular as a consequence of Theorem 2.3 in \cite{MasRig} we deduce the following
 \begin{theorem}\label{TH_conseq_th}
    Let $N = M^m \times_u F^k$ a complete Einstein warped product with Einstein constant $\lambda < 0$, warping function $u = e^{-f/k}$ and Einstein fibre $F^k$ with Einstein constant $\mu < 0$. Suppose that
    \begin{equation}\label{4_eq_condflambdamu}
      f \geq \frac k2 \log\pa{\frac{\lambda}{2\mu}} \qquad \text{for all } x \in M
    \end{equation}
    and that
    \begin{equation}
      \abs{f} \leq D\pa{1 + r(x)}^\nu
    \end{equation}
    for some $D \geq 0, \, \nu \in \mathbb{R}$. Then $N$ is simply a Riemannian product(up to a rescaling of the metric on $F$), provided
    \begin{equation}
      0 \leq \nu < 1.
    \end{equation}
    \end{theorem}
       \begin{proof}
      Since $N$ is an Einstein warped product, from the previous discussions we know that $f$ satisfies \eqref{EinsteinConstants}. Now, referring to Theorem 2.3 in \cite{MasRig}, condition (2.23) is satisfied (with equality sign) for $\delta=0$ and $\lambda = -(n-1)H^2 = -(m+k-1)H^2$, condition (2.25) is guaranteed by \eqref{4_eq_condflambdamu} and (2.26) is valid for all $\theta \in \mathbb{R}$, since $A=B=e^{-f}$, so we can choose, for instance, $\theta = -2$. Hence $f$ is constant by Theorem 2.3 in \cite{MasRig}.
    \end{proof}
    \begin{remark}
      Note that condition \eqref{4_eq_condflambdamu} is more general than \eqref{condflambdamu}, but in order to obtain triviality in Theorem \ref{TH_conseq_th} we also need to require $f$ to have sublinear growth.
    \end{remark}

\section{A refined version of Theorem 1 in \cite{rimoldi1}}\label{3}

In the present section we state a weighted version of Theorem 1.31 in \cite{prsmax}, which can be proved by minor changes to the proof of this latter, and a sufficient condition for the validity of the full Omori-Yau maximum principle for the $f$-Laplacian; our goal is to deduce a triviality result for complete Einstein warped products, which is a corollary of Theorem 1 in \cite{rimoldi1}, replacing the integrability assumption with weight $e^{-\frac{f}{k}}$ in the aforementioned theorem with a more natural condition.
We recall that a Riemannian manifold $\pa{M, \pair{ , }}$ is said to \emph{satisfy the Omori-Yau maximum principle for the $f$-laplacian} if for each $u \in C^2(M)$ such that $u^* = \sup_M u < +\infty$ there exists a sequence $\set{x_k} \subset M$ such that
\[
(i) \, u(x_k) > u^* - \frac 1k, \quad (ii) \, \abs{\nabla u(x_k)} < \frac 1k, \quad (iii) \, \Delta_f u(x_k) < \frac 1k
\]
for each $k \in \mathbb{N}$.

\begin{theorem}\label{Motomiya}Assume on the complete weighted manifold $(M, g_M, e^{-f}d\rm{vol})$ the validity of the full Omori-Yau maximum principle for the $f$-Laplacian. Let $v\in C^2(M)$ be a solution of the differential inequality
\[
\ \Delta_f v\geq \Phi(v,\abs{\nabla v}),
\]
with $\Phi(t,y)$ continuous in $t$, $C^2$ in $y$ and such that
\[
\ \frac{\partial \Phi}{\partial y}(t,y)\geq 0.
\]
Set $\varphi(t)=\Phi(t,0)$. Then a sufficient condition to guarantee that
\[
\ v^*=\sup_M v<+\infty
\]
is the existence of a continuous function $F$ positive on $[a,+\infty)$ for some $a\in\mathbb{R}$ satisfying
\begin{equation}
\left\{\int_a^tF(s)ds\right\}^{-\frac{1}{2}}\in L^1(+\infty),
\end{equation}
\begin{equation}
\limsup_{t\rightarrow+\infty}\frac{\int_a^tF(s)ds}{tF(t)}<+\infty,
\end{equation}
\begin{equation}
\liminf_{t\to+\infty}\frac{\varphi(t)}{F(t)}>0
\end{equation}
and
\begin{equation}
\liminf_{t\to+\infty}\frac{\left\{\int_a^tF(s)ds\right\}^{\frac{1}{2}}}{F(t)}\left.\frac{\partial \Phi}{\partial y}\right|_{(t,0)}>-\infty.
\end{equation}
Furthermore in this case
\[
\ \varphi(v^*)\leq 0.
\]
\end{theorem}
Consider now the equation
\begin{equation}\label{lambdamu}
\Delta_f f=k\lambda-k\mu e^{\frac{2}{k}f}.
\end{equation}
and let $\mu<0$. If we choose $\varphi(t)=\Phi(t, y)=m\lambda-m\mu e^{\frac{2}{k}t}$  and $F(t)=(t-a)^{\sigma}$, with $\sigma>1$, then $F$ satisfies the assumptions of Theorem \ref{Motomiya}. However, to use Theorem \ref{Motomiya}, we have also to assure on $(M, g_M, e^{-f}d\rm{vol})$ the validity of the full Omori-Yau maximum principle for the $f$-laplacian.

We will use the following corollary of Theorem 4.1 in \cite{prrims}.
\begin{corollary}
Let $(M^m, g_M, e^{-f}d\rm{vol})$ be a complete weighted manifold such that
\begin{equation}\label{Ricfm_lowbnd}
Ric_f^k(\nabla r, \nabla r)\geq -(m+k-1)G(r)
\end{equation}
for a smooth positive function $G$ on $[0,+\infty)$, even at the origin and satisfying
\begin{equation}\label{hp5}
\begin{array}{lll}
&\left(i\right)\,G\left(0\right)>0&\left(ii\right)\,G^{\prime}\left(t\right)\geq 0 \textrm{\,\,on\,\,} \left[0,+\infty\right)\\
&\left(iii\right)G\left(t\right)^{-\frac{1}{2}}\notin L^{1}\left(+\infty\right)&\left(iv\right)\, \limsup_{t\rightarrow+\infty}\frac{tG\left(t^{\frac{1}{2}}\right)}{G\left(t\right)}<+\infty.
\end{array}
\end{equation}
Then the Omori-Yau maximum principle for the $f$-Laplacian holds on $M$.
\end{corollary}
\begin{proof}
Let $h$ be the solution on $\mathbb{R}_0^+$ of the Cauchy problem
\[
\ \left\{
\begin{array}{ll}h^{\prime\prime}-Gh=0\\
h(0)=0;\,\,h^{\prime}(0)=1.
\end{array}\right.
\]
Then, by Proposition 2.3 in \cite{mrs}, the inequality
\[
\ \Delta_f r\leq-(m+k-1)\frac{h^{\prime}}{h}\leq C_1G(r)^{\frac{1}{2}},
\]
holds pointwise in $M\setminus(cut(o)\cup\{o\})$ for some constant $C_1$. Thus
\begin{equation}\label{f-lap-gamma}
\Delta_f r^2=2r\Delta_fr+2\leq 2+2rC_1G(r)^{\frac{1}{2}}\leq C_2rG(r)^{\frac{1}{2}},
\end{equation}
off a compact set, and the hypotheses (4.1), (4.2) and (4.3) of Theorem 4.1 in \cite{prrims} are satisfied with $\gamma=r^2$. In that theorem it is also assumed a bound on the gradient of $f$, but here we don't need this further hypothesis. Indeed by \eqref{f-lap-gamma} we can replace the last part of the proof of Theorem 4.1 in \cite{prrims} with the following computation. 
\begin{align*}
\Delta_{f}u&\left(x_{k}\right)=\Delta u\left(x_{k}\right)
-\left\langle \nabla u,\nabla f\right\rangle\left(x_{k}\right)\\
\leq&\frac{(u\left(x_{k}\right)-u\left(p\right)+1)}{k}
\left\{\frac{\varphi^\prime(\gamma(x_k))}{\varphi(\gamma(x_k))}\Delta\gamma(x_k)+\frac{1}{k}\left(\frac{\varphi^\prime(\gamma(x_k))}{\varphi(\gamma(x_k))}\right)^2\left|\nabla\gamma\right|^2(x_k)\right\}\\
&-\frac{(u\left(x_{k}\right)-u\left(p\right)+1)}{k}\frac{\varphi^\prime(\gamma(x_k))}{\varphi(\gamma(x_k))}\left\langle \nabla \gamma (x_k), \nabla f(x_k)\right\rangle\\
\leq&\frac{(u\left(x_{k}\right)-u\left(p\right)+1)}{k}
\left\{\frac{\varphi^\prime(\gamma(x_k))}{\varphi(\gamma(x_k))}\Delta_f\gamma(x_k)+\frac{1}{k}\left(\frac{\varphi^\prime(\gamma(x_k))}{\varphi(\gamma(x_k))}\right)^2\left|\nabla\gamma\right|^2(x_k)\right\}\\
\leq&\frac{(u\left(x_{k}\right)-u\left(p\right)+1)}{k}
\left\{\frac{c}{\gamma^{1/2}G\left(\gamma^{1/2}\right)^{1/2}}C_2\gamma^{1/2}
G\left(\gamma^{1/2}\right)^{1/2}\right.\\
&\left.+\frac{1}{k}\cdot\frac{c^{2}}{\gamma G\left(\gamma^{1/2}\right)}A^{2}
\gamma\right\},
\end{align*}
and the RHS tends to zero as $k\rightarrow+\infty$.
\end{proof}
Hence, choosing $G(t)=t^2+\frac{|\lambda|+\varepsilon}{m+k-1}$, for some $\varepsilon>0$, we obtain that the full Omori--Yau maximum principle for the $f$--laplacian holds on a generic quasi--Einstein manifold. 

As an application of Theorem \ref{Motomiya} we can deduce the following result.

\begin{corollary}\label{eqvol}
Let $N^{m+k}=M^{m}\times_{u}F^{k}$ be a complete Einstein warped product with non-positive scalar curvature $(m+k)\lambda=\, ^{N}S\leq0$, warping function $u(x)=e^{-\frac{f\left(  x\right)  }{k}}$ satisfying $\inf_M f=f_*>-\infty$ and complete Einstein fibre $F$. Suppose also that $^FS<0$. Then $f^*<+\infty$. In particular Riemannian volumes are equivalent to $f$--weighted volumes.
\end{corollary}
\begin{proof}
Applying Theorem \ref{Motomiya} to equation \eqref{lambdamu} we obtain that $f^*<+\infty$. Since, by assumption, we know also that $f_*>-\infty$ the thesis follows easily.
\end{proof}

From Corollary \ref{eqvol} we immediately get the following corollary of Theorem 1 in \cite{rimoldi1}.
\begin{corollary}\label{intchange}
Let $N^{m+k}=M^{m}\times_{u}F^{k}$ be a complete Einstein warped product with non-positive scalar curvature $(m+k)\lambda=\, ^{N}S\leq0$, warping function $u(x)=e^{-\frac{f\left(  x\right)  }{k}}$ satisfying $\inf_M f=f_*>-\infty$ and complete Einstein fibre $F$. Suppose also that $^FS<0$. Then $N$ is simply a Riemannian product if the base manifold $M$ is complete and non-compact, the warping function satisfies $f\in L^p(M, e^{-f}d\rm{vol})$, for
some $1<p<+\infty$, and $f\left(  x_{0}\right)  \leq0$ for some point $x_{0}\in M$.
\end{corollary}

From the Motomiya--type theorem we deduce also the following result.
\begin{theorem}
Let $N^{m+k}=M^m\times_uF^k$ be a complete Einstein warped product with non--positive scalar curvature $(m+k)\lambda=\,^NS\leq 0$, warping function $u(x)=e^{-\frac{f(x)}{k}}$ satisfying $\inf_M f=f_*>-\infty$ and complete Einstein fibre $F$ with $\,^FS<0$. Then $\,^MS_*=m\lambda$.
\end{theorem}
\begin{proof}
As above, by Theorem \ref{Motomiya}, we have that $f^{*}<+\infty$ and so
\[
\ vol_{\hat{f}}(M)\leq vol_f(M)e^{\frac{k-1}{k}f^{*}}
\]
From the weighted volume estimates in \cite{qian} and Theorem 9 in \cite{prims} we get that the weak maximum principle at infinity for the $\hat{f}$-Laplacian holds on $M$. Hence we can construct a sequence $\{x_n\}$ such that $f(x_n)\to f_*$ and $\Delta_{\hat{f}}f(x_n)\geq-\frac{1}{n}$. Thus, since tracing \eqref{QE} we have that $\Delta_{\hat{f}}f=m\lambda-\,^MS$, we obtain that
\[
\ -\frac{1}{n}\leq m\lambda-\,^MS(x_n)\leq m\lambda-\,^MS_{*}\leq0,
\]
where in the last inequality we have used the estimates of Theorem 3 in \cite{rimoldi1}.
The conclusion now follows taking the limit for $n\to+\infty$.
\end{proof}

\section{A Liouville result for 1-quasi-Einstein manifolds}\label{4}
In this section we obtain another Liouville result for $k$-quasi-Einstein manifolds, in case $k=1$. It is well known that any 1-quasi-Einstein metric which has $\mu=0$ (and so corresponds to a warped product Einstein metric) necessarily has constant scalar curvature $R\equiv(m-1)\lambda$. This follows simply by taking the trace of the quasi-Einstein equation \eqref{QE} and using equation \eqref{EinsteinConstants}.  Warped product Einstein metrics which correspond to these latters are more commonly known as \emph{static metrics} and have been studied extensively due to their connections to scalar curvature, the positive mass theorem, and general relativity, (see e.g. \cite{An}, \cite{Co} and references indicated in the recent preprint \cite{HePetWy}). 

As observed in \cite{caseCWM}, also the study of quasi-Einstein metrics with $k=1$ and $\mu\neq0$ is interesting. Since we cannot apply Theorem \ref{EwpandQEcharacterization} to construct the related Einstein warped products, their existence proves that, even restricting to integer hidden dimension $k$, quasi--Einstein manifolds form a strictly larger class of manifolds that those which are the base of an Einstein warped product manifold. For some examples of these manifolds, constructed in the more general setting of conformally warped manifolds, see the last section of \cite{caseCWM}.

Our Liouville result, which is relevant exactly in the $\mu\neq0$ case, will follow from an adaptation to the $f$-Laplacian under weighted volume growth conditions of Theorem A in \cite{PRS_JFA1}. This can be deduced from the proof of the latter, making minor modifications in the proofs of Theorem A, Lemma 1.2, Theorem A$^{\prime}$ in \cite{PRS_GAFA} and Theorem 2.5 in \cite{PRS_JFA1}.
\begin{theorem}\label{thJFA_f}
Let $\phi$ be a continuous function on $[0,+\infty)$ satisfying the conditions
\begin{equation}\label{phi}
\begin{array}{llll}
&\left(i\right)\,\phi\left(0\right)=\phi(a)=0,&\left(ii\right)\,\phi\left(s\right)> 0\textrm{\,\,in\,\,} \left(0,a\right),&\left(iii\right)\,\phi(s)<0 \textrm{\,\,in\,\,} \left(a,+\infty\right),
\end{array}
\end{equation}
for some $a>0$, and
\begin{equation}\label{phi1}
\liminf_{s\to+\infty}\frac{-\phi(s)}{s^{\sigma}}>0,
\end{equation}
for some $\sigma>1$; let also $b(x)\in C^0(M)$ and suppose that
\begin{equation*}
b(x)\geq \frac{C}{(1+r(x))^\mu}\textrm{\,\,on\,\,} M,
\end{equation*}
for some $C>0$ and $0\leq\mu<2$. Let $u$ be a non-negative solution of
\begin{equation}\label{eq_JFA_f}
\Delta_f u=-b(x)\phi(u) \textrm{\,\,on\,\,} M.
\end{equation}
Assume that
\begin{equation}\label{vol_fgrowth}
\liminf_{r\to+\infty}\frac{\log vol_f(B_r)}{r^{2-\mu}}<+\infty
\end{equation}
and, if
\begin{equation*}
(vol_f(\partial B_r))^{-1}\in L^1(+\infty)
\end{equation*}
assume furthermore that
\begin{equation*}
\phi(t)\geq ct^{\xi}\qquad0<t\ll1
\end{equation*}
for some $\xi>0$ and $c>0$. Finally, if $\xi\geq1$ suppose also that
\begin{equation*}
u(x)\geq Dr(x)^{-\theta},\qquad r(x)\gg 1
\end{equation*}
for some $\theta\geq 0$, $D>0$ and that
\begin{equation*}
\liminf_{r\to+\infty}\frac{\log vol_f(B_r)}{r^{2-\theta(\xi-1+\varepsilon)-\mu}}<+\infty
\end{equation*}
for some $\varepsilon>0$. Then $u$ is constant and identically equal to $0$ or $a$.
\end{theorem}

Now, consider a $k$-quasi-Einstein manifold $(M^m,g_M,e^{-f}d\rm{vol})$, $k<+\infty$. We recall that, according to Lemma 4 in \cite{rimoldi1}, setting $\tilde{f}=\frac{k+2}{k}f$, the scalar curvature $S$ of a quasi-Einstein manifold satisfies the following relation,
\begin{equation}\label{eq_Scal1}
\frac{1}{2}\Delta_{\tilde{f}} S= -\frac{k-1}{k}\abs{Ric-\frac{1}{m}Sg_M}^2-\frac{k+m-1}{km}\left(S-m\lambda\right)\left(S-\frac{m(m-1)}{k+m-1}\lambda\right).
\end{equation}
Exploiting \eqref{eq_Scal1}, one can obtain estimates for the infimum of the scalar curvature $S_*=\inf_MS$ (see Theorem 3 in \cite{rimoldi1}). In particular we have that for $\lambda>0$
\begin{equation}\label{estscal1}
\frac{m(m-1)}{m+k-1}<S_*\leq m\lambda,
\end{equation}
and for $\lambda<0$ and $\inf_M f>-\infty$
\begin{equation}\label{estscal2}
m\lambda\leq S_*\leq\frac{m(m-1)}{m+k-1}\lambda.
\end{equation}
In the special case $k=1$ equation \eqref{eq_Scal1} becomes
\begin{equation}\label{eq_Scal2}
\Delta_{\tilde{f}} S= -2(S-m\lambda)(S-(m-1)\lambda).
\end{equation}
Making an essential use of \eqref{eq_Scal2}, we now prove the following theorem which, jointly with \eqref{estscal1} and \eqref{estscal2}, essentially states that, under suitable geometric assumption, when the scalar curvature is confined in a particular interval it has to be constant and identically equal to one of the extremes of the interval. Some extra rigidity in case $\lambda>0$ is also discussed.
\begin{theorem}
Let $(M^m,g_M,e^{-f}d\rm{vol})$ be a geodesically complete $1$- quasi-Einstein manifold with quasi-Einstein constant $\lambda$ and scalar curvature $S$. Set $\tilde{f}=3f$ and suppose that $f_*=\inf_Mf>-\infty$. \\
If
\begin{equation*}
(vol_{\tilde{f}}(\partial B_r))^{-1}\in L^1(+\infty),
\end{equation*}
letting
\begin{equation*}
u(x)=
\begin{cases}
-S(x)+(m-1)\lambda& \lambda<0\\
-S(x)+m\lambda&\lambda>0,
\end{cases}
\end{equation*}
assume furthermore that
\begin{equation*}
u(x)\geq Dr(x)^{-\theta},\qquad r(x)\gg 1
\end{equation*}
for some $\theta\geq 0$, $D>0$, and that
\begin{equation*}
\liminf_{r\to+\infty}\frac{\log vol_{\tilde{f}}(B_r)}{r^{2-\theta\varepsilon}}<+\infty
\end{equation*}
for some $\varepsilon>0$.
\begin{enumerate}
\item[(a)] If $\lambda<0$ and $S\leq(m-1)\lambda$ we obtain that $S$ is constant and identically equal to either $(m-1)\lambda$ or $m\lambda$.
\item[(b)] If $\lambda>0$ and $S\leq m\lambda$ then $S$ is constant, identically equal to $m\lambda$ and $M$ is Einstein.
\end{enumerate}

\end{theorem}
\begin{proof}
(a) Assume $\lambda<0$. Considering $u=-S+(m-1)\lambda$, which is non-negative for $S\leq(m-1)\lambda$, from \eqref{eq_Scal2} we obtain that
\begin{equation}\label{eq_u1}
\Delta_{\tilde{f}}u=2u(u+\lambda).
\end{equation}
We want now to apply Theorem \ref{thJFA_f} to the equation \eqref{eq_u1} on the weighted manifold $(M,g_M, e^{-\tilde{f}}d\rm{vol})$. If we choose $\phi(t)=-2t(t+\lambda)$, it clearly satisfies assumptions \eqref{phi} and \eqref{phi1} with $a=-\lambda$ and the equation \eqref{eq_u1} can be written in the form
\begin{equation*}
\Delta_{\tilde{f}}u=-\phi(u),
\end{equation*}
as in the statement of Theorem \ref{thJFA_f}.

Furthermore, according to Qian weighted volume estimates, (\cite{qian}), since by assumption $f_*>-\infty$, we have the validity of the condition on the $\tilde{f}$- volume growth of the form \eqref{vol_fgrowth}.
Hence by Theorem \ref{thJFA_f} we are able to conclude that $S$ is constant and identically equal to either $m\lambda$ or $(m-1)\lambda$.

(b) Assume $\lambda>0$. Consider $u=-S+ m\lambda$ which is non-negative for $S\leq m\lambda$, and choose $\phi(t)=-2t(t-\lambda)$; applying Theorem \ref{thJFA_f} with $a=\lambda$, we conclude that $S$ is constant and identically equal to either $(m-1)\lambda$ or $m\lambda$.

Now we show that the first case cannot happen. Indeed, suppose that $S\equiv (m-1)\lambda$. Substituing in the trace of the quasi-Einstein equation we get that $\Delta f\geq 0$ and since $M$ is compact we obtain that $f$ is constant and $M$ is Einstein with $Ric=\lambda g_M$. But this is clearly impossible, since tracing this latter equation we get a contradiction. Hence $S\equiv m\lambda$. Substituing again in the trace of the quasi-Einstein equation we obtain, reasoning as above, that $f$ is constant, and thus that $M$ is Einstein. 
\end{proof}
\begin{acknowledgements*}
We wish to thank Stefano Pigola and Jeffrey Case for valuable suggestions and useful comments on earlier versions of the paper.
\end{acknowledgements*}

\bibliographystyle{amsplain}
\bibliography{ByblioQE}

\providecommand{\bysame}{\leavevmode\hbox to3em{\hrulefill}\thinspace}
\providecommand{\MR}{\relax\ifhmode\unskip\space\fi MR }
\providecommand{\MRhref}[2]{%
  \href{http://www.ams.org/mathscinet-getitem?mr=#1}{#2}
}
\providecommand{\href}[2]{#2}
\begin{thebibliography}{10}

\bibitem{An}
M.~T. Anderson, \emph{Scalar curvature, metric degenerations and the static
  vacuum {E}instein equations on 3-manifolds}, Geom. Funct. Anal. \textbf{9}
  (1999), no.~2, 855--967.

\bibitem{Besse}
A.~Besse, \emph{Einstein manifolds. {R}eprint of the 1997 edition}, Classics in
  Mathematics, Springer-Verlag, Berlin, 2008.

\bibitem{calabi}
E.~Calabi, \emph{An extension of {H}opf's maximum principle with an application
  to {R}iemannian geometry}, Duke Math. J. \textbf{25} (1957), 45--56.

\bibitem{caseCWM}
J.~S. Case, \emph{Conformally warped manifolds and quasi-{E}instein metrics},
  arXiv:1011.2723v1 [math.DG] (2010).

\bibitem{case}
\bysame, \emph{On the nonexistence of quasi-{E}instein metrics. {T}o appear on
  \emph{Pacific J. Math.}}, arXiv:0902.2226v3 [math.DG] (2010).

\bibitem{caseshuwei}
J.~S. Case, Y.-J. Shu, and G.~Wei, \emph{Rigidity of quasi-{E}instein metrics},
  arXiv:0805.3132v1 [math.DG] (2008).

\bibitem{Co}
J.~Corvino, \emph{Scalar curvature deformation and gluing construction for the
  {E}instein constraint equations}, Comm. Math. Phys. \textbf{214} (2000),
  no.~1, 137--189.

\bibitem{ELnM}
M.~Eminenti, G.~La Nave, and C.~Mantegazza, \emph{Ricci solitons: the equation
  point of view}, Manuscripta Math. \textbf{127} (2008), 345--367.

\bibitem{hamilton}
R.~S. Hamilton, \emph{The {R}icci flow on surfaces. \emph{Mathematics and
  general relativity (Santa Cruz, CA, 1986)}}, Contemp. Math., vol.~71,
  pp.~237--262, Am. Math. Soc., 1988.

\bibitem{HePetWy}
C.~He, P.~Petersen, and W.~Wylie, \emph{On the classification of warped product
  {E}instein metrics}, arXiv:1010.5488v1 [math.DG] (2010).

\bibitem{kimkim}
D.-S. Kim and Y.~H. Kim, \emph{Compact {E}instein warped product spaces with
  nonpositive scalar curvature}, Proc. Amer. Math. Soc. \textbf{131} (2003),
  2573--2576.

\bibitem{li}
X.-D. Li, \emph{Liouville theorems for symmetric diffusion operators on
  complete {R}iemannian manifolds}, J. Math. Pures Appl. \textbf{84} (2005),
  1295--1361.

\bibitem{LuPagePope}
H.~Lu, D.~N. Page, and C.~N. Pope, \emph{New inhomogeneous {E}instein metrics
  on sphere bundles over {E}instein-{K}aehler manifolds}, Phys. Lett. B
  \textbf{593} (2004), 218--226.

\bibitem{mrs}
L.~Mari, M.~Rigoli, and A.~G. Setti, \emph{{K}eller-{O}sserman conditions for
  diffusion-type operators on {R}iemannian manifolds}, J. Funct. Anal.
  \textbf{258} (2010), 665--712.

\bibitem{MasRig}
P.~Mastrolia and M.~Rigoli, \emph{Diffusion-type operators, {L}iouville
  theorems and gradient estimates on complete manifolds}, Nonlinear Anal.
  \textbf{72} (2010), 3767--3785.

\bibitem{perelman}
G.~Perelman, \emph{Ricci flow with surgery on three manifolds},
  arXiv:math/0303109v1 [math.DG] (2003).

\bibitem{petwylie1}
P.~Petersen and W.~Wylie, \emph{On gradient {R}icci solitons with symmetry},
  Proc. Amer. Math. Soc. \textbf{137} (2009), 2085--2092.

\bibitem{petwylie3}
\bysame, \emph{On the classification of gradient {R}icci solitons},
  arXiv:0712.1298v5 [math.DG] (2009).

\bibitem{petwylie2}
\bysame, \emph{Rigidity of gradient {R}icci solitons}, Pacific J. Math.
  \textbf{241} (2009), 329--345.

\bibitem{prrims}
S.~Pigola, M.~Rigoli, M.~Rimoldi, and A.~G. Setti, \emph{Ricci almost
  solitons.}, {T}o appear on {A}nn. {S}c. {N}orm. {S}up. {P}isa.
  arXiv:1003.2945v1 (2010).

\bibitem{PRS_GAFA}
S.~Pigola, M.~Rigoli, and A.~G. Setti, \emph{Volume growth, ``a priori''
  estimates, and geometric applications}, Geom. Funct. Anal. \textbf{13}
  (2003), no.~6, 1302--1328.

\bibitem{PRS_JFA1}
\bysame, \emph{A {L}iouville-type result for quasi-linear elliptic equations on
  complete {R}iemannian manifolds}, J. Funct. Anal. \textbf{219} (2005), no.~2,
  400--432.

\bibitem{prsmax}
\bysame, \emph{Maximum principles on {R}iemannian manifolds and applications},
  Memoirs of the AMS, vol. 174, 2005.

\bibitem{prims}
S.~Pigola, M.~Rimoldi, and A.~G. Setti, \emph{Remarks on non-compact gradient
  {R}icci solitons}, To appear on Math. Z. arXiv:0905.2868v3 [mathDG] (2010).

\bibitem{qian}
Z.~Qian, \emph{Estimates for weighted volumes and applications}, Quart. J.
  Math. Oxford \textbf{48} (1997), 235--242.

\bibitem{rimoldi1}
M.~Rimoldi, \emph{A remark on {E}instein warped product. {T}o appear on
  \emph{Pacific J. Math.}}, arXiv:1004.3866v3 [math.DG] (2010).

\bibitem{SchoenYau}
R.~Schoen and S.-T. Yau, \emph{Lectures on differential geometry}, Conference
  Proceedings and Lecture Notes in Geometry and Topology, I, International
  Press, Cambridge, 1994.

\bibitem{WeiWylie2proc}
G.~Wei and W.~Wylie, \emph{Comparison {G}eometry for the {S}mooth {M}etric
  {M}easure {S}paces}, Proceedings of the 4th {I}nternational {C}ongress of
  {C}hinese {M}athematicians, vol.~II, Hangzhou, China, 2007, pp.~191--202.

\bibitem{WeiWylie1}
G.~Wei and W.~Wylie, \emph{Comparison geometry for the {B}akry-{E}mery {R}icci
  tensor}, J. Differential Geom. \textbf{83} (2009), no.~2, 377--405.
  \MR{2577473}

\end{thebibliography}

\end{document}